\documentclass[a4paper,10pt, reqno]{amsart}
\usepackage[backref]{hyperref}
\usepackage{a4wide}
\usepackage{amsthm}
\usepackage{amssymb}
 \newtheorem{proposition}{Proposition}
 \newtheorem{theorem}{Theorem}
 
 \newtheorem{lemma}{Lemma}

\begin{document}
\title{The large sieve with square norm moduli in $\mathbb{Z}[i]$}

\author{Stephan Baier}
\address{Stephan Baier, Indian Institute of Science Education and Research Thiruvananthapuram, College of Engineering Trivandrum Campus,
Trivandrum - 695016, Kerala, India}

\email{sbaier@math.tifr.res.in}


\maketitle

\section{Introduction}
The classical large sieve inequality asserts that 
\begin{equation*} 
\sum\limits_{q\le Q} \sum\limits_{\substack{a=1\\ (a,q)=1}}^{q} \left|\sum\limits_{n\le N} a_n \cdot e\left(n\cdot \frac{a}{q}\right)\right|^2 \ll \left(Q^2+N\right)\sum\limits_{n\le N} |a_n|^2
\end{equation*}
for any $Q,N\ge 1$ and any sequence $(a_n)_{n\in \mathbb{N}}$ of complex numbers. (Equivalently, the summation of $n$ over the interval $(0,N]$ can be replaced by a summation over any interval $(M,M+N]$.) The large sieve with square moduli was investigated by L. Zhao and the author of the present paper in a series of papers (see \cite{Baie}, \cite{BaZh}, \cite{Zhao}). To date, the best result is the following.
\begin{equation*} 
\sum\limits_{q\le Q} \sum\limits_{\substack{a=1\\ (a,q)=1}}^{q^2} \left|\sum\limits_{n\le N} a_n \cdot e\left(n\cdot \frac{a}{q^2}\right)\right|^2 \ll (QN)^{\varepsilon}\left(Q^3+\min\left\{Q^2\sqrt{N},\sqrt{Q}N\right\}+N\right)\sum\limits_{n\le n} |a_n|^2,
\end{equation*}
where $\varepsilon$ is any positive constant, and the implied $\ll$-constants depends only on $\varepsilon$. The object of this paper is to establish a large sieve inequality for square norm moduli in $\mathbb{Z}[i]$. 

M. Huxley \cite{Huxl} established a generalization of the large sieve for number fields, which we describe in the following. Let $K$ be an algebraic number field of degree $k$ over $\mathbb{Q}$ and let $(\theta_1,...,\theta_k)$ be an integral basis of $K$, so that every integer $\xi$ of $K$ is representable uniquely as 
$$
\xi=n_1\theta_1+...+n_k\theta_k,
$$
where $n_1,...,n_k$ are rational integers. For any integral ideal $\mathfrak{a}$ of $K$, let $\mathcal{N}(\mathfrak{a})$ be its norm and $\sigma(\xi)$ an additive character modulo $\mathfrak{a}$. Such a character is called proper if it is not an additive character modulo an ideal $\mathfrak{b}$ which divides $\mathfrak{a}$ properly. Then for any $X,N_1,...,N_k\ge 1$, $M_1,...,M_k\in \mathbb{R}$ and complex sequence $(b_n)_{n\in \mathbb{Z}^k}$, we have
\begin{equation} \label{huxley}
\begin{split}
& \sum\limits_{\mathcal{N}(\mathfrak{a})\le X^k} \sum\limits_{\substack{\sigma \bmod{\mathfrak{a}}\\ \sigma\ \mbox{\scriptsize proper}}} \left|\sum\limits_{M_1<n_1\le M_1+N_1}\cdots
\sum\limits_{M_k<n_k\le M_k+N_k} b_{n_1,...,n_k} \cdot \sigma(n_1\xi_1+\cdots+n_k\xi_k)\right|^2\\ 
\ll & \prod\limits_{j=1}^k \left(N_j^{1/2}+X\right)^2 \cdot 
\sum\limits_{M_1<n_i\le M_1+N_1}\cdots
\sum\limits_{M_k<n_k\le M_k+N_k} |b_{n_1,...,n_k}|^2,
\end{split}
\end{equation}
where the implied $\ll$-constant depends only on the field $K$. As demonstrated in \cite{Huxl}, using Gauss sums similarly as in the case $K=\mathbb{Q}$, the above can be converted into a large sieve inequality for multiplicative characters $\chi$ of the form
\begin{equation*} 
\begin{split}
& \sum\limits_{\mathcal{N}(\mathfrak{a})\le X^k} \frac{\mathcal{N}(\mathfrak{a})}{\Phi(\mathfrak{a})}\cdot \sum\limits_{\substack{\chi \bmod{\mathfrak{a}}\\ \chi\ \mbox{\scriptsize proper}}} \left|\sum\limits_{M_1<n_1\le M_1+N_1}\cdots
\sum\limits_{M_k<n_k\le M_k+N_k} b_{n_1,...,n_k} \cdot \chi(n_1\xi_1+\cdots+n_k\xi_k)\right|^2\\ 
\ll & \prod\limits_{j=1}^k \left(N_j^{1/2}+X\right)^2 
\sum\limits_{M_1<n_i\le M_1+N_1}\cdots
\sum\limits_{M_k<n_k\le M_k+N_k} |b_{n_1,...,n_k}|^2,
\end{split}
\end{equation*}
where $\Phi(\mathfrak{a})$ is the generalized Euler totient function for ideals in $K$, and, similarly as in the case of additive characters, $\chi$ is called proper if it is not a multiplicative character modulo an ideal $\mathfrak{b}$ which divides $\mathfrak{a}$ properly.
In $\mathbb{Z}[i]$, which is a principal ideal domain, the proper additive characters for the ideal $(q)$ take the form
$$
\sigma(\xi)=e\left(\mbox{Tr}\left(\frac{\xi r}{2q}\right)\right)=e\left(\Re\left(\frac{\xi r}{q}\right)\right),
$$
where $r\in \mathbb{Z}[i]$ ranges over a reduced residue system modulo $q$ (in particular, $r$ and $q$ are coprime). In the above, $\mbox{Tr}(x)$ denotes the trace of $x\in \mathbb{Z}[i]$, given by $\mbox{Tr}(x)=x+\overline{x}=2\Re x$. Hence, setting $K:=\mathbb{Q}[i]$, $k=2$, $\xi_1=1$, $\xi_2=i$, $X:=Q$, $N_j:=\sqrt{N}$ and $a_n:=b_{\Re n,\Im n}$ for $n\in \mathbb{Z}[i]$, where $b_{\Re n,\Im n}=0$ if $\mathcal{N}(n)>N$, we deduce the following version of the large sieve for $\mathbb{Z}[i]$ from \eqref{huxley}. 

\begin{theorem} \label{theo1} Let $Q,N\ge 1$ and $(a_n)_{n\in \mathbb{Z}[i]}$ be any sequence of complex numbers. Then 
$$
\sum\limits_{\substack{q\in \mathbb{Z}[i]\setminus\{0\}\\ \mathcal{N}(q)\le Q}} \sum\limits_{\substack{r \bmod{q}\\ (r,q)=1}} \left|\sum\limits_{\substack{n\in \mathbb{Z}[i]\\ \mathcal{N}(n)\le N}}  a_n \cdot e\left(\mbox{\rm Tr}\left(\frac{nr}{2q}\right)\right)\right|^2 \ll \left(Q^2+N\right)\sum\limits_{\substack{n\in \mathbb{Z}[i]\\ \mathcal{N}(n)\le N}} |a_n|^2.
$$
\end{theorem}

Here, as in the following, $\mathcal{N}(x)$ denotes the norm of $x\in \mathbb{Z}[i]$, given by $\mathcal{N}(x)=x\overline{x}=(\Re x)^2+(\Im x)^2$, and $r$ runs over a reduced residue system modulo $q$ in $\mathbb{Z}[i]$, $(r,q)=1$ indicating the coprimality of $r$ and $q$. 
A version of the large sieve for $\mathbb{Z}[i]$ with moduli confined to natural numbers was proved by W. Schlackow in \cite[Theorem 4.2.1]{Schl} and may be reformulated as follows. 

\begin{theorem} \label{theo3} Let $Q,N\ge 1$ and $(a_n)_{n\in \mathbb{Z}[i]}$ be any sequence of complex numbers. Then 
$$
\sum\limits_{\substack{q\in \mathbb{N}\\ q\le Q}} \sum\limits_{\substack{r \bmod{q}\\ (r,q)=1}} \left|\sum\limits_{\substack{n\in \mathbb{Z}[i]\\ \mathcal{N}(n)\le N}}  a_n \cdot e\left(\mbox{\rm Tr}\left(\frac{nr}{2q}\right)\right)\right|^2 \ll \left(Q^3+Q^2\sqrt{N}+N\right)\sum\limits_{\substack{n\in \mathbb{Z}[i]\\ \mathcal{N}(n)\le N}} |a_n|^2,
$$
where $r$ runs over a reduced residue system modulo $q$ in $\mathbb{Z}[i]$. 
\end{theorem}  

In this paper, we shall establish the following version of the large sieve with square norm moduli for $\mathbb{Z}[i]$. 

\begin{theorem} \label{theo5} Let $Q,N\ge 1$ and $(a_n)_{n\in \mathbb{Z}[i]}$ be any sequence of complex numbers. Then 
$$
\sum\limits_{\substack{q\in \mathbb{Z}[i]\setminus\{0\}\\ \mathcal{N}(q)\le Q^2\\ \mathcal{N}(q)=\Box}} \sum\limits_{\substack{r \bmod{q}\\ (r,q)=1}} \left|\sum\limits_{\substack{n\in \mathbb{Z}[i]\\ \mathcal{N}(n)\le N}}  a_n \cdot e\left(\mbox{\rm Tr}\left(\frac{nr}{2q}\right)\right)\right|^2 \ll (QN)^{\varepsilon}\left(Q^3+Q^2\sqrt{N}+\sqrt{Q}N\right)\sum\limits_{\substack{n\in \mathbb{Z}[i]\\ \mathcal{N}(n)\le N}} |a_n|^2,
$$
where $\varepsilon$ is any positive constant, and the implied $\ll$-constant depends only on $\varepsilon$. 
\end{theorem}

Here, as in the following, $\mathcal{N}(q)=\Box$ indicates that $\mathcal{N}(q)$ is a perfect square. The following analogue for multiplicative characters can be deduced by the standard procedure using Gauss sums mentioned above. 

\begin{theorem} Let $Q,N\ge 1$ and $(a_n)_{n\in \mathbb{Z}[i]}$ be any sequence of complex numbers. Then 
$$
\sum\limits_{\substack{q\in \mathbb{Z}[i]\setminus\{0\}\\ \mathcal{N}(q)\le Q^2\\ \mathcal{N}(q)=\Box}} \frac{\mathcal{N}(q)}{\Phi(q)} \cdot \sum\limits_{\substack{\chi \bmod{(q)}\\ \chi \ \mbox{\scriptsize \rm proper}}} \left|\sum\limits_{\substack{n\in \mathbb{Z}[i]\\ \mathcal{N}(n)\le N}}  a_n\cdot  \chi(n)\right|^2 \ll (QN)^{\varepsilon}\left(Q^3+Q^2\sqrt{N}+\sqrt{Q}N\right)\sum\limits_{\substack{n\in \mathbb{Z}[i]\\ \mathcal{N}(n)\le N}} |a_n|^2,
$$
where $\varepsilon$ is any positive constant, and the implied $\ll$-constant depends only on $\varepsilon$.
\end{theorem}

We start by considering sums over restricted sets of moduli of the form
\begin{equation} \label{Sigmadef}
\Sigma(Q,N;\mathcal{S}):=\sum\limits_{\substack{q\in \mathcal{S}\\ Q/2<\mathcal{N}(q)\le Q}} \sum\limits_{\substack{r \bmod{q}\\ (r,q)=1}} \left| \sum\limits_{\substack{n\in \mathbb{Z}[i]\\ \mathcal{N}(n)\le N}}a_{n}\cdot e\left(\mbox{Tr}\left(\frac{nr}{2q}\right)\right)\right|^2,
\end{equation}
where $\mathcal{S}$ is a subset of $\mathbb{Z}[i]$. The restriction to moduli norms in dyadic intervals will be of importance in our method.
To estimate the above sums, we first use the double large sieve due to H. Iwaniec and E. Bombieri \cite{IwBo}. This will lead us to a lattice point counting problem, which we reformulate as counting certain points in a disk. Considerations about the spacing of these points and the Poisson summation formula will enable us to recover slightly weakened versions of  Theorems \ref{theo1} and Theorem \ref{theo3}. This weakening by a factor of logarithm comes from our restriction to dyadic intervals above. The main point of this paper is to prove Theorem \ref{theo5}, for which we shall, in addition to the above-mentioned spacing results and the Poisson summation formula, use a 2-dimensional Weyl shift similar to the 1-dimensional Weyl shift performed in \cite{Zhao}. 

We note that Theorem \ref{theo5} is the precise analogue to the first known large sieve inequality for square moduli in $\mathbb{Z}$, proved by L. Zhao  \cite{Zhao}, which asserts that 
\begin{equation*} 
\sum\limits_{q\le Q} \sum\limits_{\substack{a=1\\ (a,q)=1}}^{q^2} \left|\sum\limits_{n\le N} a_n \cdot e\left(n\cdot \frac{a}{q^2}\right)\right|^2 \ll (QN)^{\varepsilon}\left(Q^3+Q^2\sqrt{N}+\sqrt{Q}N\right)\sum\limits_{n\le N} |a_n|^2.
\end{equation*}

Throughout this paper, we follow the usual convention that $\varepsilon$ is an arbitrarily small positive number that can change from line to line, and $O$-constants may depend on $\varepsilon$. \\

{\bf Acknowledgement.} The author thanks the IISER Thiruvananthapuram for financial support and excellent working conditions. 

\section{Initial transformations}
In this section, we start with some initial transformations of the sum $\Sigma(Q,N;\mathcal{S})$ defined in \eqref{Sigmadef}. Setting 
$$
q=u+vi,\quad  r=x+yi, \quad n=s+ti,
$$
a notation which we shall use throughout the sequel, we deduce that  
\begin{equation} \label{1}
\begin{split}
\Sigma(Q,N;\mathcal{S}) = & \sum\limits_{\substack{q\in \mathcal{S}\\ Q/2<\mathcal{N}(q)\le Q}} \sum\limits_{\substack{r \bmod{q}\\ (r,q)=1}} \left|\sum\limits_{\substack{n\in \mathbb{Z}[i]\\ \mathcal{N}(n)\le N}}
a_{n}\cdot e\left(\frac{(sx-ty)u+(sy+tx)v}{\mathcal{N}(q)}\right)\right|^2\\
= &  \sum\limits_{\substack{q\in \mathcal{S}\\ Q/2<\mathcal{N}(q)\le Q}} \sum\limits_{\substack{r \bmod{q}\\ (r,q)=1}} \left|\sum\limits_{\substack{n\in \mathbb{Z}[i]\\ \mathcal{N}(n)\le N}} a_n\cdot e\left(\left(\frac{xu+yv}{\mathcal{N}(q)},\frac{xv-yu}{\mathcal{N}(q)}\right)\cdot (s,t)\right)\right|^2\\
= & \sum\limits_{\substack{q\in \mathcal{S}\\ Q/2<\mathcal{N}(q)\le Q}} \sum\limits_{\substack{r \bmod{q}\\ (r,q)=1}} \left|\sum\limits_{|s|\le \sqrt{N}}\sum\limits_{|t|\le \sqrt{N}} a_{s,t}'\cdot e\left(\left(\frac{xu+yv}{\mathcal{N}(q)},\frac{xv-yu}{\mathcal{N}(q)}\right)\cdot (s,t)\right)\right|^2,
\end{split}
\end{equation}
where 
\begin{equation} \label{ast'}
a_{s,t}':=
\begin{cases}
a_{s+it} & \mbox{ if } \mathcal{N}(s+it)\le N, \\
0 & \mbox{ if } \mathcal{N}(s+it)>N.  
\end{cases}
\end{equation}
We set 
$$
\mathcal{T}(q,r):=\sum\limits_{|s|\le \sqrt{N}}\sum\limits_{|t|\le \sqrt{N}} a_{s,t}'\cdot e\left(\left(\frac{xu+yv}{\mathcal{N}(q)},\frac{xv-yu}{\mathcal{N}(q)}\right)\cdot (s,t)\right)
$$
and 
$$
b_{q,r}:= \begin{cases} 
\left|\mathcal{T}(q,r)\right|^2/\mathcal{T}(q,r) & \mbox{ if } \mathcal{T}(q,r)\not=0, \\ 0 & \mbox{ otherwise.}
\end{cases}
$$
Then it follows that
\begin{equation} \label{from}
\Sigma(Q,N;\mathcal{S})= \sum\limits_{\substack{q\in \mathcal{S}\\ Q/2<\mathcal{N}(q)\le Q}} \sum\limits_{\substack{r \bmod{q}\\ (r,q)=1}} \sum\limits_{|s|\le \sqrt{N}} \sum\limits_{|t|\le \sqrt{N}} a_{s,t}'b_{q,r} \cdot e\left(\left(\frac{xu+yv}{\mathcal{N}(q)},\frac{xv-yu}{\mathcal{N}(q)}\right)\cdot (s,t)\right)
\end{equation}
and also
\begin{equation} \label{hallo}
\Sigma(Q,N;\mathcal{S})=  \sum\limits_{\substack{q\in \mathcal{S}\\ Q/2<\mathcal{N}(q)\le Q}} \sum\limits_{\substack{r \bmod{q}\\ (r,q)=1}}\left| b_{q,r} \right|^2.
\end{equation}

\section{Application of the double large sieve}
Now we use the double large sieve due to Bombieri and Iwaniec \cite{IwBo} to further estimate $\Sigma(Q,N;\mathcal{S})$. 

\begin{lemma}[Bombieri-Iwaniec] \label{dls} Let ${\bf X}$ and ${\bf Y}$ be two subsets of $\mathbb{R}^K$. Let $a({\bf x})$ and $b({\bf y})$ be arbitrary complex numbers for ${\bf x}\in {\bf X}$ and ${\bf y} \in {\bf Y}$. Let $X_1,...,X_K,Y_1,...,Y_K$ be positive numbers. Define the bilinear forms
\begin{equation*}
\begin{split}
B(b;{\bf X}):=&\mathop{\sum\limits_{{\bf y}\in {\bf Y}} \sum\limits_{{\bf y}'\in {\bf Y}}}_{|{\bf y}-{\bf y}'|< (2X_k)^{-1};\ k=1,...,K} |b({\bf y})b({\bf y}')|,\\
B(a;{\bf Y}):=&\mathop{\sum\limits_{{\bf x}\in {\bf X}} \sum\limits_{{\bf x}'\in {\bf X}}}_{|{\bf x}-{\bf x}'|< (2Y_k)^{-1};\ k=1,...,K} |a({\bf x})a({\bf x}')|,\\
B(a,b;{\bf X},{\bf Y}):= &\sum\limits_{\substack{{\bf x}\in {\bf X}\\ |x_k|< X_k}} \sum\limits_{\substack{{\bf y}\in {\bf Y}\\ |y_k|< Y_k}} a({\bf x})b({\bf y})e({\bf x}\cdot {\bf y}).
\end{split}
\end{equation*}
Then
$$
|B(a,b;{\bf X},{\bf Y})|^2\le \left(2\pi^2\right)^K \prod\limits_{k=1}^K \left(1+X_kY_k\right)B(b;{\bf X})B(a;{\bf Y}).
$$
\end{lemma}  

We note that in the original statement of the above lemma, "$< (2X_k)^{-1}$", "$< (2Y_k)^{-1}$", "$|x_k|< X_k$" and "$|y_k|< Y_k$" were replaced by  "$\le (2X_k)^{-1}$", "$\le (2Y_k)^{-1}$", "$|x_k|\le X_k$" and "$|y_k|\le Y_k$", respectively, but the proof in \cite{IwBo} applies to the above variant as well. 

In the following let, for any real number $z$, $\{z\}=z-[z]$ be its fractional part, $||z||$ its distance to the nearest integer, and 
\begin{equation} \label{fdef}
f(z):=\left\{z+\frac{1}{2}\right\}-\frac{1}{2}. 
\end{equation}
Then noting that 
$$
e\left(\left(\frac{xu+yv}{\mathcal{N}(q)},\frac{xv-yu}{\mathcal{N}(q)}\right)\cdot (s,t)\right)=e\left(\left(f\left(\frac{xu+yv}{\mathcal{N}(q)}\right),f\left(\frac{xv-yu}{\mathcal{N}(q)}\right)\right)\cdot (s,t)\right)
$$
and using \eqref{ast'} and Lemma \ref{dls} with
\begin{equation*}
\begin{split}
{\bf X}:= & \left\{(s,t)\in \mathbb{Z}^2 \ : \ \mathcal{N}(s+it)\le N\right\},\\
{\bf Y}:= & \left\{\left(f\left(\frac{xu+yv}{\mathcal{N}(q)}\right),f\left(\frac{xv-yu}{\mathcal{N}(q)}\right)\right)\in \mathbb{R}^2\ :\ q\in \mathcal{S},\ Q/2<\mathcal{N}(q)\le Q,\ r \bmod{q},\  (r,q)=1\right\}, \\
X_1:= & \sqrt{N}, \quad X_2:=\sqrt{N}, \quad Y_1:=\frac{1}{2}, \quad Y_2:=\frac{1}{2},
\end{split}
\end{equation*}
we obtain
\begin{equation*}
\begin{split}
\left| \Sigma(Q,N;\mathcal{S})\right|^2 \ll N \cdot \mathop{\sum\limits_{\substack{q_1\in \mathcal{S}\\ Q/2<\mathcal{N}(q_1)\le Q}} \sum\limits_{\substack{r_1 \bmod{q_1}\\ (r_1,q_1)=1}}  \sum\limits_{\substack{q_2\in \mathcal{S}\\ Q/2<\mathcal{N}(q_2)\le Q}} \sum\limits_{\substack{r_2 \bmod{q_2}\\ (r_2,q_2)=1}}}_{\substack{\left|\left|(x_1u_1+y_1v_1)/\mathcal{N}(q_1)-(x_2u_2+y_2v_2)/\mathcal{N}(q_2)\right|\right|\le 1/\sqrt{N}\\ \left|\left|(x_1v_1-y_1u_1)/\mathcal{N}(q_1)-(x_2v_2-y_2u_2)/\mathcal{N}(q_2)\right|\right|\le 1/\sqrt{N}}} \left|b_{r_1,q_1}b_{r_2,q_2}\right| \cdot \sum\limits_{\substack{n\in \mathbb{Z}[i]\\ \mathcal{N}(n)\le N}} \left| a_n\right|^2
\end{split}
\end{equation*}
from \eqref{from}, where we write
$$
q_j=u_j+v_ji\quad \mbox{and} \quad  r_j=x_j+y_ji \quad \mbox{for } j=1,2.
$$
Since
\begin{equation*}
\left|b_{r_1,q_1}b_{r_2,q_2}\right|\le \left|b_{r_1,q_1}\right|^2+ \left|b_{r_2,q_2}\right|^2,
\end{equation*}
it follows that
\begin{equation*}
\begin{split}
& \left| \Sigma(Q,N;\mathcal{S})\right|^2\\  \ll & NZ \cdot \max\limits_{\substack{\substack{q_2\in \mathcal{S}\\ Q/2<\mathcal{N}(q_2)\le Q\\  r_2 \bmod q_2\\ (r_2,q_2)=1}}}  \sharp\Bigg\{(q_1,r_1)\ : \ q_1\in \mathcal{S}, \ Q/2<\mathcal{N}(q_1)\le Q, \ r_1 \bmod{q_1},\ (r_1,q_1)=1, \\ & \left|\left|\frac{x_1u_1+y_1v_1}{\mathcal{N}(q_1)}-\frac{k}{\mathcal{N}(q_2)}\right|\right|\le \frac{1}{\sqrt{N}}, \ 
\left|\left|\frac{x_1v_1-y_1u_1}{\mathcal{N}(q_1)}-\frac{l}{\mathcal{N}(q_2)}\right|\right|\le \frac{1}{\sqrt{N}} \Bigg\}\times\\ & \sum\limits_{\substack{q\in \mathcal{S}\\ Q/2<\mathcal{N}(q)\le Q}} \sum\limits_{\substack{r \bmod{q}\\ (r,q)=1}}  \left|b_{q,r}\right|^2,
\end{split}
\end{equation*}
where we set 
$$
Z:= \sum\limits_{\substack{n\in \mathbb{Z}[i]\\ \mathcal{N}(n)\le N}} \left| a_{n}\right|^2
$$
and 
\begin{equation} \label{kl}
k:=x_2u_2+y_2v_2\quad \mbox{and} \quad l:=x_2v_2-y_2u_2
\end{equation}
throughout the sequel. Using \eqref{hallo} and dividing both sides by $|\Sigma(Q,N;\mathcal{S})|$, we deduce that 
\begin{equation*}
\begin{split}
\Sigma(Q,N;\mathcal{S})
\ll & NZ \cdot  \max\limits_{\substack{q_2\in \mathcal{S}\\ Q/2<\mathcal{N}(q_2)\le Q\\ r_2 \bmod q_2\\ (r_2,q_2)=1}}  \sharp\Bigg\{(q_1,r_1)\ :\ q_1\in \mathcal{S}, \ Q/2<\mathcal{N}(q_1)\le Q, \ r_1 \bmod{q_1},\ (r_1,q_1)=1,\\  &
\left|\left|\frac{x_1u_1+y_1v_1}{\mathcal{N}(q_1)}-\frac{k}{\mathcal{N}(q_2)}\right|\right|\le \frac{1}{\sqrt{N}}, \
\left|\left|\frac{x_1v_1-y_1u_1}{\mathcal{N}(q_1)}-\frac{l}{\mathcal{N}(q_2)}\right|\right|\le \frac{1}{\sqrt{N}} \Bigg\}.
\end{split}
\end{equation*}

\section{Reduction to a lattice point counting problem}
\subsection{Simplification of the problem} If
$$
\left|\left|\frac{x_1u_1+y_1v_1}{\mathcal{N}(q_1)}-\frac{k}{\mathcal{N}(q_2)}\right|\right|\le \frac{1}{\sqrt{N}}, \quad \mbox{and} \quad 
\left|\left|\frac{x_1v_1-y_1u_1}{\mathcal{N}(q_1)}-\frac{l}{\mathcal{N}(q_2)}\right|\right|\le \frac{1}{\sqrt{N}},
$$
then we can find $(x_1',y_1')\in \mathbb{Z}^2$ such that 
\begin{equation} \label{2}
\left|\left|\frac{x_1u_1+y_1v_1}{\mathcal{N}(q_1)}-\frac{k}{\mathcal{N}(q_2)}\right|\right|=\left|\frac{x_1'u_1+y_1'v_1}{\mathcal{N}(q_1)}-\frac{k}{\mathcal{N}(q_2)}\right|
\end{equation}
and
\begin{equation} \label{3}
\left|\left|\frac{x_1v_1-y_1u_1}{\mathcal{N}(q_1)}-\frac{l}{\mathcal{N}(q_2)}\right|\right|= \left|\frac{x_1'v_1-y_1'u_1}{\mathcal{N}(q_1)}-\frac{l}{\mathcal{N}(q_2)}\right|.
\end{equation}
Indeed, if  
$$
x_1'=x_1+au_1+bv_1 \quad \mbox{and} \quad y_1'=y_1+av_1-bu_1,
$$
then
$$
\frac{x_1'u_1+y_1'v_1}{\mathcal{N}(q_1)}=a+\frac{x_1u_1+y_1v_1}{\mathcal{N}(q_1)} \quad \mbox{and} \quad \frac{x_1'v_1-y_1'u_1}{\mathcal{N}(q_1)}=b+\frac{x_1v_1-y_1u_1}{\mathcal{N}(q_1)},
$$
and thus for suitable $a,b\in \mathbb{Z}$, we get \eqref{2} and \eqref{3}. It follows that 
\begin{equation*}
\begin{split}
\Sigma(Q,N;\mathcal{S})
\ll & NZ\cdot   
\max\limits_{\substack{q_2\in \mathcal{S}\\ Q/2<\mathcal{N}(q_2)\le Q\\ r_2 \bmod q_2\\ (r_2,q_2)=1}}  \sharp\Bigg\{(q_1,x_1',y_1')\ : \ q_1\in \mathcal{S}, \ Q/2<\mathcal{N}(q_1)\le Q,\ \binom{x_1'}{y_1'}\in \mathbb{Z}^2,\\
& \left|\frac{x_1'u_1+y_1'v_1}{\mathcal{N}(q_1)}-\frac{k}{\mathcal{N}(q_2)}\right|\le \frac{1}{\sqrt{N}},\ 
\left|\frac{x_1'v_1-y_1'u_1}{\mathcal{N}(q_1)}-\frac{l}{\mathcal{N}(q_2)}\right|\le \frac{1}{\sqrt{N}} \Bigg\}.\end{split}
\end{equation*}
In the following, for brevity of notation, we shall replace $(q_1,u_1,v_1,x_1',y_1')$ by $(q,u,v,x,y)$. 

\subsection{Rescaling and rotating}
We may interpret the above as a problem of counting points of orthogonal lattices in a closed square because the last estimate is equivalent to
\begin{equation*}
\begin{split}
\Sigma(Q,N;\mathcal{S}) \ll  NZ\cdot  
\max\limits_{\substack{q_2\in \mathcal{S}\\ Q/2<\mathcal{N}(q_2)\le Q\\ r_2 \bmod q_2\\ (r_2,q_2)=1}} &  \sharp\left\{ (q,x,y)\ :\ q\in \mathcal{S}, \ Q/2<\mathcal{N}(q)\le Q, \ \binom{x}{y}\in \mathbb{Z}^2,\right.  \\ &  \left. 
x\binom{u/\mathcal{N}(q)}{v/\mathcal{N}(q)}+y\binom{v/\mathcal{N}(q)}{-u/\mathcal{N}(q)}\in S(k,l)\right\},
\end{split}
\end{equation*}
where $S(k,l)$ is the closed square defined by
$$
S(k,l):=\left[\frac{k}{\mathcal{N}(q_2)}-\frac{1}{\sqrt{N}},\frac{k}{\mathcal{N}(q_2)}+\frac{1}{\sqrt{N}}\right]\times \left[\frac{l}{\mathcal{N}(q_2)}-\frac{1}{\sqrt{N}},\frac{l}{\mathcal{N}(q_2)}+\frac{1}{\sqrt{N}}\right].
$$ 

Rescaling by a factor of $\sqrt{\mathcal{N}(q)}$, it follows that 
\begin{equation} \label{ditte}
\begin{split}
\Sigma(Q,N;\mathcal{S}) \ll  NZ\cdot  
\max\limits_{\substack{q_2\in \mathcal{S}\\ Q/2<\mathcal{N}(q_2)\le Q\\ r_2 \bmod q_2\\ (r_2,q_2)=1}}  & \sharp\left\{(q,x,y)\ :\ q\in \mathcal{S}, \ Q/2<\mathcal{N}(q)\le Q, \ \binom{x}{y}\in \mathbb{Z}^2, \right.\\
& \left.
x\binom{u'}{v'}+y\binom{v'}{-u'}\in \sqrt{\mathcal{N}(q)}S(k,l)\right\},
\end{split}
\end{equation}
where $\sqrt{\mathcal{N}(q)}S(k,l)$ is the square
$$
\sqrt{\mathcal{N}(q)} S(k,l):=\left\{\sqrt{\mathcal{N}(q)}{\bf r} \ :\ {\bf r}\in S(k,l)\right\},
$$
and 
$$
\binom{u'}{v'}=\frac{1}{\sqrt{\mathcal{N}(q)}}\binom{u}{v} \quad \mbox{and} \quad \binom{v'}{-u'}=\frac{1}{\sqrt{\mathcal{N}(q)}}\binom{v}{-u}
$$
are orthonormal vectors. 

For two real numbers $\mu$ and $\nu$ set 
$$
M(\mu,\nu):=\begin{pmatrix} \mu & \nu \\ -\nu & \mu \end{pmatrix}.
$$
Then rotating by applying the rotation matrix 
$$
M(u',v')=\begin{pmatrix} u' & v' \\ -v' & u' \end{pmatrix}=\frac{1}{\sqrt{\mathcal{N}(q)}} \begin{pmatrix} u & v \\ -v & u \end{pmatrix}=\frac{1}{\sqrt{\mathcal{N}(q)}}M(u,v),
$$
and replacing $y$ by $-y$, we deduce from \eqref{ditte} that
\begin{equation*}
\begin{split}
& \Sigma(Q,N;\mathcal{S}) 
\ll NZ\times\\ & \max\limits_{\substack{q_2\in \mathcal{S}\\ Q/2<\mathcal{N}(q_2)\le Q\\ r_2 \bmod q_2\\ (r_2,q_2)=1}}  \sharp\left\{(q,x,y)\ :\  q\in \mathcal{S}, \ Q/2<\mathcal{N}(q)\le Q, \ \binom{x}{y}\in  
\mathbb{Z}^2\cap M(u,v)S(k,l)\right\},
\end{split}
\end{equation*}
where 
$$
M(u,v)S(k,l):=\left\{M(u,v){\bf s}\ :\ {\bf s}\in S(k,l)\right\}
$$
is the square $\sqrt{\mathcal{N}(q)}S(k,l)$, rotated by applying the matrix $M(u',v')$.

\subsection{Switching between lattices}
The closed square $M(u,v)S(k,l)$ is contained in the closed disk with radius 
\begin{equation} \label{radius}
R:=\sqrt{\frac{4Q}{N}} 
\end{equation}
and midpoint
$$
M(u,v)\binom{k/\mathcal{N}(q_2)}{l/\mathcal{N}(q_2)}=u\binom{k/\mathcal{N}(q_2)}{l/\mathcal{N}(q_2)}+v\binom{l/\mathcal{N}(q_2)}{-k/\mathcal{N}(q_2)}.
$$
Therefore, we have
\begin{equation*}
\begin{split}
\Sigma(Q,N;\mathcal{S}) 
\ll NZ\cdot 
 \max\limits_{\substack{q_2\in \mathcal{S}\\ Q/2<\mathcal{N}(q_2)\le Q\\ r_2 \bmod q_2\\ (r_2,q_2)=1}} & \sharp\Bigg\{  (q,x,y)\ :\  q\in \mathcal{S}, \ Q/2<\mathcal{N}(q)\le Q, \\
& \binom{x}{y}\in \mathbb{Z}^2\cap D_{R}\left(u\binom{k/\mathcal{N}(q_2)}{l/\mathcal{N}(q_2)}+v\binom{l/\mathcal{N}(q_2)}{-k/\mathcal{N}(q_2)}\right)\Bigg\},
\end{split}
\end{equation*}
where $D_R({\bf r})$ is the closed disk with radius $R$ and midpoint ${\bf r}$.  
Hence, we count points of the standard lattice $\mathbb{Z}^2$ contained in closed $R$-neighborhoods of points of the lattice
$$
\mathcal{L}:=\left(\tilde{x}\binom{k/\mathcal{N}(q_2)}{l/\mathcal{N}(q_2)}+\tilde{y}\binom{l/\mathcal{N}(q_2)}{-k/\mathcal{N}(q_2)}\right)_{\binom{\tilde{x}}{\tilde{y}}\in \mathbb{Z}^2}.
$$
If $R\ge 1/2$, then there are $O(R^2)$ $\mathbb{Z}^2$-points in the closed $R$-neighborhood of every $\mathcal{L}$-point, and using \eqref{radius}, it therefore follows that
\begin{equation} \label{Rlarge}
\Sigma(Q,N;\mathcal{S}) 
\ll QZ\cdot \sharp \left\{q\in \mathcal{S} \ :\  Q/2<\mathcal{N}(q)\le Q\right\}.
\end{equation}

In the following, we assume that $R<1/2$. We observe that counting $\mathbb{Z}^2$-points in closed $R$-neighborhoods of $\mathcal{L}$-points
amounts to the same as counting $\mathcal{L}$-points  in closed $R$-neighborhoods of $\mathbb{Z}^2$-points. By this switch of lattices, we have 
\begin{equation} \label{das}
\begin{split}
\Sigma(Q,N;\mathcal{S}) 
\ll NZ\cdot
\max\limits_{\substack{q_2\in \mathcal{S}\\ Q/2<\mathcal{N}(q_2)\le Q\\ r_2 \bmod q_2\\ (r_2,q_2)=1}} & \sharp\Bigg\{  (q,x,y)\ :\  q\in \mathcal{S}, \ Q/2<\mathcal{N}(q)\le Q, \ \binom{x}{y}\in \mathbb{Z}^2,\\ &
u\binom{k/\mathcal{N}(q_2)}{l/\mathcal{N}(q_2)}+v\binom{l/\mathcal{N}(q_2)}{-k/\mathcal{N}(q_2)}\in D_R\binom{x}{y}\Bigg\}.
\end{split}
\end{equation}
Since $R<1/2$, \eqref{das} is equivalent to 
\begin{equation} \label{smallR}
\begin{split}
& \Sigma(Q,N;\mathcal{S}) \ll NZ\times\\ &    
\max\limits_{\substack{q_2\in \mathcal{S}\\ Q/2<\mathcal{N}(q_2)\le Q\\ r_2 \bmod q_2\\ (r_2,q_2)=1}}  \sharp\left\{q\in \mathcal{S} \ : \ Q/2<\mathcal{N}(q)\le Q, \    
\binom{f\left((uk+vl)/\mathcal{N}(q_2)\right)}{f\left((-vk+ul)/\mathcal{N}(q_2)\right)}\in D_R({\bf 0})\right\},
\end{split}
\end{equation}
where $f(z)$ is defined as in \eqref{fdef}.
We combine \eqref{Rlarge} and \eqref{smallR} below.

\begin{proposition} \label{lab} Let ${\bf 1}_I$ be the indicator function of the interval $I$. Then we have 
\begin{equation} \label{ex}
\begin{split}
& \Sigma(Q,N;\mathcal{S}) \ll QZ\cdot \sharp \left\{q\in \mathcal{S} \ :\  Q/2<\mathcal{N}(q)\le Q\right\}+NZ\cdot {\bf 1}_{(0,1/2)}(R) \times\\ &    
\max\limits_{\substack{q_2\in \mathcal{S}\\ Q/2<\mathcal{N}(q_2)\le Q\\ r_2 \bmod q_2\\ (r_2,q_2)=1}}  \sharp\left\{q\in \mathcal{S} \ : \ Q/2<\mathcal{N}(q)\le Q, \    
\binom{f\left((uk+vl)/\mathcal{N}(q_2)\right)}{f\left((-vk+ul)/\mathcal{N}(q_2)\right)}\in D_R({\bf 0})\right\}.
\end{split}
\end{equation}
\end{proposition}

\section{Reproof of a slightly weakened version of Theorem 1}
In this section, we reprove Theorem 1 in a slightly weakened form, namely we establish the inequality
\begin{equation} \label{ine}
\sum\limits_{\substack{q\in \mathbb{Z}[i]\setminus\{0\}\\ \mathcal{N}(q)\le Q}} \sum\limits_{\substack{r \bmod{q}\\ (r,q)=1}} \left|\sum\limits_{\substack{n\in \mathbb{Z}[i]\\ \mathcal{N}(n)\le N}}  a_n \cdot e\left(\mbox{\rm Tr}\left(\frac{nr}{2q}\right)\right)\right|^2 \ll \left(Q^2+N\log(2Q)\right)Z.
\end{equation}

\subsection{Spacing modulo 1}
To estimate the maximum on the right-hand side of \eqref{ex}, we prove the following lemma on the spacing of the points 
$$
\binom{f\left((uk+vl)/\mathcal{N}(q_2)\right)}{f\left((-vk+ul)/\mathcal{N}(q_2)\right)}.
$$

\begin{lemma} \label{spacinglemma} Assume that $u_2+v_2i=q_2\in \mathbb{Z}[i]\setminus\{0\}$, $x_2+y_2i=r_2\in \mathbb{Z}[i]$, $(r_2,q_2)=1$, $k,l$ are given as in \eqref{kl}, and $u,v,\tilde{u},\tilde{v}\in \mathbb{Z}$. Then 
\begin{equation*}
\begin{split}
& \left|f\left(\frac{uk+vl}{\mathcal{N}(q_2)}\right)-f\left(\frac{\tilde{u}k+\tilde{v}l}{\mathcal{N}(q_2)}\right)\right|^2 +\left|f\left(\frac{-vk+ul}{\mathcal{N}(q_2)}\right)-f\left(\frac{-\tilde{v}k+\tilde{u}l}{\mathcal{N}(q_2)}\right)\right|^2 \\ & 
\begin{cases} \ge 1/\mathcal{N}(q_2) & \mbox{ if } q_2\nmid\left((u-\tilde{u})+(v-\tilde{v})i\right),\\ 0 & \mbox{ if } q_2|\left((u-\tilde{u})+(v-\tilde{v})i\right).
\end{cases}
\end{split}
\end{equation*}
\end{lemma}

\begin{proof} Clearly, there exist integers $\mu$ and $\nu$ such that 
$$
f\left(\frac{uk+vl}{\mathcal{N}(q_2)}\right)-f\left(\frac{\tilde{u}k+\tilde{v}l}{\mathcal{N}(q_2)}\right)=\frac{u_0k+v_0l}{\mathcal{N}(q_2)}-\mu 
$$
and 
$$
f\left(\frac{-vk+ul}{\mathcal{N}(q_2)}\right)-f\left(\frac{-\tilde{v}k+\tilde{u}l}{\mathcal{N}(q_2)}\right)= \frac{-v_0k+u_0l}{\mathcal{N}(q_2)}-\nu,
$$
where $u_0:=u-\tilde{u}$ and $v_0:=v-\tilde{v}$. It follows that
\begin{equation*}
\begin{split}
& \left|f\left(\frac{uk+vl}{\mathcal{N}(q_2)}\right)-f\left(\frac{\tilde{u}k+\tilde{v}l}{\mathcal{N}(q_2)}\right)\right|^2 +\left|f\left(\frac{-vk+ul}{\mathcal{N}(q_2)}\right)-f\left(\frac{-\tilde{v}k+\tilde{u}l}{\mathcal{N}(q_2)}\right)\right|^2\\ 
= & \left(\frac{u_0k+v_0l}{\mathcal{N}(q_2)}-\mu\right)^2+\left(\frac{-v_0k+u_0l}{\mathcal{N}(q_2)}-\nu\right)^2\\
= & \frac{(u_0^2+v_0^2)(k^2+l^2)}{\mathcal{N}(q_2)^2}-2\left(\mu\cdot \frac{u_0k+v_0l}{\mathcal{N}(q_2)}+\nu \cdot \frac{-v_0k+u_0l}{\mathcal{N}(q_2)}\right)+\mu^2+\nu^2.
\end{split}
\end{equation*}
From \eqref{kl}, we have
\begin{equation} \label{and}
k^2+l^2=(x_2u_2+y_2v_2)^2+(x_2v_2-y_2u_2)^2=(x_2^2+y_2^2)(u_2^2+v_2^2)=(x_2^2+y_2^2)\mathcal{N}(q_2).
\end{equation}
Hence,
\begin{equation*}
\begin{split}
& \left|f\left(\frac{uk+vl}{\mathcal{N}(q_2)}\right)-f\left(\frac{\tilde{u}k+\tilde{v}l}{\mathcal{N}(q_2)}\right)\right|^2 +\left|f\left(\frac{-vk+ul}{\mathcal{N}(q_2)}\right)-f\left(\frac{-\tilde{v}k+\tilde{u}l}{\mathcal{N}(q_2)}\right)\right|^2\\ 
= & \frac{(u_0^2+v_0^2)(x_2^2+y_2^2)-2\left(\mu(u_0k+v_0l)+\nu(-v_0k+u_0l)\right)}{\mathcal{N}(q_2)}+\mu^2+\nu^2.
\end{split}
\end{equation*}
By non-negativity, this equals zero or is greater or equal to $1/\mathcal{N}(q_2)$. Further, we have 
\begin{equation*}
\left|f\left(\frac{uk+vl}{\mathcal{N}(q_2)}\right)-f\left(\frac{\tilde{u}k+\tilde{v}l}{\mathcal{N}(q_2)}\right)\right|^2 +\left|f\left(\frac{-vk+ul}{\mathcal{N}(q_2)}\right)-f\left(\frac{-\tilde{v}k+\tilde{u}l}{\mathcal{N}(q_2)}\right)\right|^2=0
\end{equation*}
if and only if 
\begin{equation} \label{system}
\begin{split}
u_0k+v_0l\equiv & 0 \bmod{\mathcal{N}(q_2)},\\
-v_0k+u_0l\equiv & 0 \bmod{\mathcal{N}(q_2)}.
\end{split}
\end{equation}
This is equivalent to $\mathcal{N}(q_2)| (u_0+v_0i)(k-li)$. Further, we observe that 
$$
k-li=(x_2+y_2i)(u_2-v_2i)=r_2\overline{q}_2.
$$
Since $(r_2,q_2)=1$, we deduce that 
\eqref{system} is equivalent to $q_2|u_0+v_0i$ and hence $q_2|(u-\tilde{u})+(v-\tilde{v})i$. This completes the proof. 
\end{proof}

\subsection{Counting points in disks} Let the conditions in Lemma \ref{spacinglemma} be satisfied. Then it follows from the same lemma that 
\begin{equation} \label{oh}
\binom{f\left((uk+vl)/\mathcal{N}(q_2)\right)}{f\left((-vk+ul)/\mathcal{N}(q_2)\right)} =
\binom{f\left((\tilde{u}k+\tilde{v}l)/\mathcal{N}(q_2)\right)}{f\left((-\tilde{v}k+\tilde{u}l)/\mathcal{N}(q_2)\right)}
\end{equation}
if and only if $q_2|\left((u-\tilde{u})+(v-\tilde{v})i\right)$. If $Q/2<\mathcal{N}(q_2)\le Q$ and $L\ge 1$, then for every $q=u+vi$ with $\mathcal{N}(q)\le Q$, the number of $\tilde{q}=\tilde{u}+\tilde{v}i$ such that
$\mathcal{N}(\tilde{q})\le LQ$ and $q_2|\left((u-\tilde{u})+(v-\tilde{v})i\right)$ is $O(L)$. (We note that this is the point where our restriction to dyadic intervals is crucial.) Therefore, again by Lemma \ref{spacinglemma}, the number of $u+vi=q\in \mathcal{S}$ such that $\mathcal{N}(q)\le LQ$ and 
$$
\binom{f\left((uk+vl)/\mathcal{N}(q_2)\right)}{f\left((-vk+ul)/\mathcal{N}(q_2)\right)}\in D_R({\bf 0})
$$
is bounded by a constant times $L$ times the maximum number of points in a disk with radius $R$ such that the distance between any two of them is greater or equal $1/\sqrt{\mathcal{N}(q_2)}$. 
This maximum number is bounded by 1 plus the maximum number of open disks with radius $1/(2\sqrt{\mathcal{N}(q_2)})$ that can be packed into a disk of radius $2R$ without overlaps. The latter is bounded by the area of a disk with radius  $2R$ divided by the area of a disk with radius $1/(2\sqrt{\mathcal{N}(q_2)})$. Altogether, we thus get the following.

\begin{proposition} \label{ballcount} Assume that $Q,L\ge 1$ and $R>0$. Then 
\begin{equation*}
\max\limits_{\substack{q_2\in \mathcal{S}\\ Q/2<\mathcal{N}(q_2)\le Q\\ r_2 \bmod q_2\\ (r_2,q_2)=1}}  \sharp\left\{q\in \mathcal{S} \ : \ \mathcal{N}(q)\le LQ, \    
\binom{f\left((uk+vl)/\mathcal{N}(q_2)\right)}{f\left((-vk+ul)/\mathcal{N}(q_2)\right)}\in D_R({\bf 0})\right\} 
\ll \left(1+\frac{R^2}{1/Q}\right)L. 
\end{equation*}
\end{proposition}

Now \eqref{radius}, Proposition \ref{lab} and Proposition \eqref{ballcount} give
\begin{equation*}
\Sigma(Q,N;\mathcal{S})\ll \left(N+Q^2\right)Z.
\end{equation*}
This holds in particular for $\mathcal{S}=\mathbb{Z}[i]$, implying \eqref{ine} upon summing up the contributions of $O(\log 2Q)$ dyadic intervals containing the moduli norms $\mathcal{N}(q)$.

\section{Reproof  of a slightly weakened version of Theorem 2} 
\subsection{General Fourier analytic approach} 
To get savings for sparse subsets $\mathcal{S}$ of $\mathbb{Z}[i]$, it may be useful to apply  Fourier analysis to estimate the right-hand side of \eqref{ex}. 
Let $\Phi\ :\ \mathbb{R}\rightarrow \mathbb{R}$ be a Schwartz class function which is positive in the interval $[1/2,1]$ and assume that $R<1/2$. Then we observe that 
\begin{equation} \label{ri}
\begin{split}
& \sharp\left\{q\in \mathcal{S} \ : \ Q/2<\mathcal{N}(q)\le Q, \    
\binom{f\left((uk+vl)/\mathcal{N}(q_2)\right)}{f\left((-vk+ul)/\mathcal{N}(q_2)\right)}\in D_R({\bf 0})\right\} \\
\ll & \sum\limits_{q\in \mathcal{S}} \Phi\left(\frac{\mathcal{N}(q)}{Q}\right)\cdot \sum\limits_{x\in \mathbb{Z}} \sum\limits_{y\in \mathbb{Z}} e^{-\pi\left(((uk+lv)/\mathcal{N}(q_2)-x)^2+((-vk+ul)/\mathcal{N}(q_2)-y)^2\right)/R^2}.
\end{split}
\end{equation}
Now the Poisson summation formula, applied to the sums over $x$ and $y$, transforms the right-hand side of \eqref{ri} into
\begin{equation} \label{4}
\begin{split}
& R^2 \cdot \sum\limits_{q\in \mathcal{S}}  \Phi\left(\frac{\mathcal{N}(q)}{Q}\right)\cdot  \sum\limits_{\alpha\in \mathbb{Z}} \sum\limits_{\beta\in \mathbb{Z}} e^{-\pi (\alpha^2 +\beta^2)R^2} 
e\left(\frac{\alpha(uk+vl)+\beta(-vk+ul)}{\mathcal{N}(q_2)}\right)\\
= & R^2 \cdot \sum\limits_{\alpha\in \mathbb{Z}} \sum\limits_{\beta\in \mathbb{Z}} e^{-\pi (\alpha^2 +\beta^2)R^2} \cdot
\sum\limits_{q\in \mathcal{S}}  \Phi\left(\frac{\mathcal{N}(q)}{Q}\right)\cdot e\left(\frac{u(\alpha k+\beta l)+v(-\beta k+\alpha l)}{\mathcal{N}(q_2)}\right).
\end{split}
\end{equation}
Using \eqref{radius}, Proposition \ref{lab} and the above, we obtain the following.

\begin{proposition} \label{Fou} We have 
\begin{equation} \label{hop} 
\begin{split}
& \Sigma(Q,N;\mathcal{S}) \ll QZ\cdot \sharp \left\{q\in \mathcal{S} \ :\  Q/2<\mathcal{N}(q)\le Q\right\}+QZ\times\\ & \max\limits_{\substack{q_2\in \mathcal{S}\\ Q/2<\mathcal{N}(q_2)\le Q\\ r_2 \bmod q_2\\ (r_2,q_2)=1}}
\sum\limits_{\alpha\in \mathbb{Z}} \sum\limits_{\beta\in \mathbb{Z}} e^{-4\pi (\alpha^2 +\beta^2)Q/N} \cdot
\sum\limits_{q\in \mathcal{S}}  \Phi\left(\frac{\mathcal{N}(q)}{Q}\right)\cdot e\left(\frac{u(\alpha k+\beta l)+v(-\beta k+\alpha l)}{\mathcal{N}(q_2)}\right). 
\end{split}
\end{equation}
\end{proposition}

For suitable sets $\mathcal{S}$, we may hope to be able to estimate the inner sum over $q$ on the right-hand side of \eqref{hop} non-trivially. We now consider the case when $\mathcal{S}=\mathbb{N}$, thus recovering a slightly weakened form of Theorem 2, namely the bound
\begin{equation}\label{iine}
\sum\limits_{\substack{q\in \mathbb{N}\\ q\le Q}} \sum\limits_{\substack{r \bmod{q}\\ (r,q)=1}} \left|\sum\limits_{\substack{n\in \mathbb{Z}[i]\\ \mathcal{N}(n)\le N}}  a_n \cdot e\left(\mbox{\rm Tr}\left(\frac{nr}{2q}\right)\right)\right|^2 \ll \left(Q^3+Q^2\sqrt{N}+N\log 2Q\right)\sum\limits_{\substack{n\in \mathbb{Z}[i]\\ \mathcal{N}(n)\le N}} |a_n|^2.
\end{equation}

We note that the above Fourier analytic approach also works if $\mathcal{S}$ equals the full set $\mathbb{Z}[i]$. In this case, the Poisson summation formula, applied to the sum over $q$, leads to a counting problem which is in a sense dual to that  considered in subsection 5.2. The resulting estimate for $\Sigma(Q,N;\mathbb{Z}[i])$ will be the same, though. For this reason, we don't carry out this calculation here.

\subsection{Case of integer moduli}
In the situation of Theorem 2 we have $\mathcal{S}=\mathbb{N}$, and the contribution of $Q/\sqrt{2}<q\le Q$ equals  
\begin{equation} \label{rupa}
\sum\limits_{\substack{q\in \mathbb{N}\\ Q/\sqrt{2}<q\le Q}} \sum\limits_{\substack{r \bmod{q}\\ (r,q)=1}} \left|\sum\limits_{\substack{n\in \mathbb{Z}[i]\\ \mathcal{N}(n)\le N}}  a_n \cdot e\left(\mbox{\rm Tr}\left(\frac{nr}{2q}\right)\right)\right|^2= \Sigma\left(Q^2,N;\mathbb{N}\right).
\end{equation}
On choosing $\Phi$ in such a way that 
$$
\Phi(z)=e^{-\pi z} \quad \mbox{ if } z\ge 0,
$$
\eqref{kl} and Proposition \ref{Fou} give
\begin{equation*}  
\begin{split}
\Sigma\left(Q^2,N;\mathbb{N}\right) \ll & Q^3Z+Q^2Z\times\\ &  \max\limits_{\substack{u_2\in \mathbb{Z}\\ Q/\sqrt{2}<u_2\le Q\\ x_2+y_2i \bmod u_2\\ (x_2+y_2i,u_2)=1}}
\sum\limits_{\alpha\in \mathbb{Z}} \sum\limits_{\beta\in \mathbb{Z}} e^{-4\pi (\alpha^2 +\beta^2)Q^2/N} \cdot 
\sum\limits_{u\in \mathbb{Z}}  e^{-\pi u^2/Q^2} \cdot e\left(\frac{u(\alpha x_2-\beta y_2)}{u_2}\right) 
\end{split}
\end{equation*}
upon noting that $v=0=v_2$ in this case. Applying the Poisson summation formula to the inner sum over $u$, we get
\begin{equation*}
\begin{split}
\sum\limits_{u\in \mathbb{Z}}  e^{-\pi u^2/Q^2} \cdot e\left(\frac{u(\alpha x_2-\beta y_2)}{u_2}\right) = & \sum\limits_{h=1}^{u_2} e\left(\frac{h(\alpha x_2-\beta y_2)}{u_2}\right)
\cdot \sum\limits_{\substack{u\in \mathbb{Z}\\ u\equiv h \bmod{u_2}}} e^{-\pi u^2/Q^2}\\
= & \frac{Q}{u_2} \cdot \sum\limits_{\gamma\in \mathbb{Z}} e^{-\pi \gamma^2(Q/u_2)^2} \cdot \sum\limits_{h=1}^{u_2} e\left(\frac{h(\alpha x_2-\beta y_2-\gamma)}{u_2}\right)\\
= & Q \cdot \sum\limits_{\substack{\gamma\in \mathbb{Z}\\ \gamma\equiv \alpha x_2-\beta y_2 \bmod{u_2}}} e^{-\pi \gamma^2(Q/u_2)^2}\\
\le & Q \cdot \sum\limits_{\substack{\gamma\in \mathbb{Z}\\ \gamma\equiv \alpha x_2-\beta y_2 \bmod{u_2}}} e^{-\pi \gamma^2} 
\end{split}
\end{equation*}
if $Q/\sqrt{2}<u_2\le Q$. 
Therefore, 
\begin{equation}  \label{z1}
\Sigma\left(Q^2,N;\mathbb{N}\right) \ll Q^3Z+Q^3Z\cdot \max\limits_{\substack{u_2\in \mathbb{Z}\\ Q/\sqrt{2}<u_2\le Q\\ x_2+y_2i \bmod u_2\\ (x_2+y_2i,u_2)=1}}
\sum\limits_{\gamma\in \mathbb{Z}} e^{-\pi \gamma^2}\cdot \mathop{\sum\limits_{\alpha\in \mathbb{Z}} \sum\limits_{\beta\in \mathbb{Z}}}_{\alpha x_2 -\beta y_2 \equiv \gamma \bmod{u_2}} e^{-4\pi (\alpha^2 +\beta^2)Q^2/N}. 
\end{equation}

We now consider the truncated sum
\begin{equation*}  
T_{\gamma}\left(U,V\right) := 
\mathop{\sum\limits_{|\alpha|\le U} \sum\limits_{|\beta|\le V}}_{\alpha x_2 -\beta y_2 \equiv \gamma \bmod{u_2}} 1. 
\end{equation*}
Let $d=(x_2,u_2)$, $x_2':=x_2/d$ and $u_2'=u_2/d$. Let $\overline{x_2'}$ be a multiplicative inverse of $x_2'$ modulo $u_2'$, i.e. $\overline{x_2'}x_2' \equiv 1 \bmod{u_2'}$. Necessarily $(d,y_2)=1$ because otherwise $(x_2+iy_2,u_2)\not=1$. Let $\overline{y_2}$ be a multiplicative inverse of $y_2$ modulo $d$, i.e. $\overline{y_2}y_2\equiv 1\bmod{d}$. It follows that 
$$
\alpha x_2 -\beta y_2 \equiv \gamma \bmod{u_2} \quad \Longleftrightarrow\quad \left(\beta\equiv - \overline{y_2}\gamma\bmod{d} \quad \mbox{and} \quad \alpha \equiv \overline{x_2'} \cdot \frac{\beta y_2+\gamma}{d} \bmod{u_2'}\right). 
$$  
This implies
\begin{equation} \label{z2}
T_{\gamma}(U,V)
\ll \left(1+\frac{V}{d}\right)\cdot \left(1+\frac{U}{u_2'}\right)\ll 1+U+V+\frac{UV}{u_2}\ll 1+U+V+\frac{UV}{Q}
\end{equation}
if $Q/\sqrt{2}<u_2\le Q$. 
Using partial summation for the sums over $\alpha$ and $\beta$ on the right-hand side of \eqref{z1} together with \eqref{z2} now gives 
\begin{equation*}
\Sigma\left(Q^2,N;\mathbb{N}\right)\ll Q^3Z \cdot \sum\limits_{\gamma\in \mathbb{Z}} e^{-\pi \gamma^2}\cdot \left(1+\frac{N^{1/2}}{Q}+\frac{N}{Q^3}\right)\ll \left(Q^3+Q^2N^{1/2}+N\right)Z,
\end{equation*}
implying \eqref{iine} upon summing up the contributions of $O(\log 2Q)$ dyadic intervals containing the moduli norms $\mathcal{N}(q)$.

\section{Proof of Theorem 3}
Now we turn to the main point of this paper, a proof of Theorem 3. In the situation of this theorem, we have $\mathcal{S}=\{q\in \mathbb{Z}[i]\ :\ \mathcal{N}(q)=\Box\}$, and the contribution of $Q^2/2<\mathcal{N}(q)\le Q^2$ equals
\begin{equation} \label{with}
\sum\limits_{\substack{q\in \mathbb{Z}[i]\setminus\{0\}\\ Q^2/2<\mathcal{N}(q)\le Q^2\\ \mathcal{N}(q)=\Box}} \sum\limits_{\substack{r \bmod{q}\\ (r,q)=1}} \left|\sum\limits_{\substack{n\in \mathbb{Z}[i]\\ \mathcal{N}(n)\le N}}  a_n \cdot e\left(\mbox{\rm Tr}\left(\frac{nr}{2q}\right)\right)\right|^2 =  \Sigma\left(Q^2,N;\mathcal{S}\right).
\end{equation}

\subsection{Case of large $Q$} We first deal with the case when $Q> N^{1/2-\varepsilon}$. For individual moduli $q\in \mathbb{Z}[i]\setminus \{0\}$, we have 
\begin{equation} \label{individual}
\sum\limits_{\substack{r \bmod{q}\\ (r,q)=1}} \left|\sum\limits_{\substack{n\in \mathbb{Z}[i]\\ \mathcal{N}(n)\le N}}  a_n \cdot e\left(\mbox{\rm Tr}\left(\frac{nr}{2q}\right)\right)\right|^2 
\ll (\mathcal{N}(q)+N)Z,
\end{equation}
which can be proved in a way analogous to the corresponding bound 
$$
\sum\limits_{\substack{r \bmod{q}\\ r\in \mathbb{Z}\\ (r,q)=1}} \left|\sum\limits_{n\le N}  a_n \cdot e\left(\frac{nr}{q}\right)\right|^2 
\ll (q+N)Z
$$ 
in the setting of rational integers. Summing up trivially over $q$ and using $Q> N^{1/2-\varepsilon}$ now gives
\begin{equation} \label{ulm}
 \Sigma\left(Q^2,N;\mathcal{S}\right) \ll Q^{3}N^{\varepsilon}Z.
\end{equation}
 
\subsection{Case of small $Q$}
In the following, we assume that $Q\le N^{1/2-\varepsilon}$. We observe that $q\in \mathcal{S}$ if and only if  $\left(u,v,\sqrt{\mathcal{N}(q)}\right)$ is a Pythagorean triple. 
Therefore, one of the numbers $u$ and $v$ is odd, and the other one is even. Without loss of generality, we may assume that $u$ is odd and $v$ is even because the contribution of the modulus $iq=-v+iu$ is the same as that of $q$. The Pythagorean triples $\left(u,v,\sqrt{\mathcal{N}(q)}\right)$ with this property are parametrized by
$$
\left(u,v,\sqrt{\mathcal{N}(q)}\right)=\left(m^2-n^2,2mn,m^2+n^2\right), \quad (m,n)\in\mathbb{Z}^2. 
$$
Thus, on choosing $\Phi$ in such a way that 
$$
\Phi(z)=e^{-\sqrt{z}} \quad \mbox{ if } z\ge 0,
$$
Proposition \ref{Fou} gives
\begin{equation*} 
\begin{split}
 \Sigma\left(Q^2,N;\mathcal{S}\right) \ll & Q^3Z+
Q^2Z\times\\ & \max\limits_{\substack{\mathcal{N}(q_2)=\Box\\ Q^2/2<\mathcal{N}(q_2)\le Q^2\\ r_2 \bmod q_2\\ (r_2,q_2)=1}} \sum\limits_{\alpha} \sum\limits_{\beta} e^{-4\pi (\alpha^2 +\beta^2)Q^2/N} \cdot
\sum\limits_{m}\sum\limits_{n} e^{-\pi \left(m^2+n^2\right)/Q} \times\\ &  e\left(\frac{(m^2-n^2)(\alpha k+\beta l)+2mn(-\beta k+\alpha l)}{\mathcal{N}(q_2)}\right).
\end{split}
\end{equation*}
Applying the  Cauchy-Schwarz inequality, we deduce that  
\begin{equation} \label{5}
\begin{split}
\left|\Sigma\left(Q^2,N;\mathcal{S}\right)\right|^2 
\ll & Q^6Z^2+Q^4Z^2\times\\ & \max\limits_{\substack{\mathcal{N}(q_2)=\Box\\ Q^2/2<\mathcal{N}(q_2)\le Q^2\\ r_2 \bmod q_2\\ (r_2,q_2)=1}}  \left(\sum\limits_{\alpha} \sum\limits_{\beta} e^{-4\pi (\alpha^2 +\beta^2)Q^2/N}\right)  \left(\sum\limits_{\alpha} \sum\limits_{\beta} e^{-4\pi (\alpha^2 +\beta^2)Q^2/N}\times \right. \\ & \left.  
\left| 
\sum\limits_{m}\sum\limits_{n} e^{-\pi (m^2+n^2)/Q} \cdot e\left(\frac{(m^2-n^2)(\alpha k+\beta l)+2mn(-\beta k+\alpha l)}{\mathcal{N}(q_2)}\right)\right|^2\right).\\
\end{split}
\end{equation}

Clearly, 
\begin{equation} \label{6}
\sum\limits_{\alpha} \sum\limits_{\beta} e^{-4\pi (\alpha^2 +\beta^2)Q^2/N}\ll 1+\frac{N}{Q^2} \ll \frac{N}{Q^2}
\end{equation} 
and 
\begin{equation} \label{7}
\begin{split}
& \sum\limits_{\alpha} \sum\limits_{\beta} e^{-4\pi (\alpha^2 +\beta^2)Q^2/N}\times\\ & \left| 
\sum\limits_{m}\sum\limits_{n} e^{-\pi (m^2+n^2)/Q} \cdot e\left(\frac{(m^2-n^2)(\alpha k+\beta l)+2mn(-\beta k+\alpha l)}{\mathcal{N}(q_2)}\right)\right|^2\\
= & \sum\limits_{\alpha} \sum\limits_{\beta} e^{-4\pi (\alpha^2 +\beta^2)Q^2/N} \cdot
\sum\limits_{m_1}\sum\limits_{n_1}\sum\limits_{m_2}\sum\limits_{n_2} e^{-\pi (m_1^2+n_1^2+m_2^2+n_2^2)/Q} \times\\ &
 e\left(\frac{(m_1^2-m_2^2-n_1^2+n_2^2)(\alpha k+\beta l)+2(m_1n_1-m_2n_2)(-\beta k+\alpha l)}{\mathcal{N}(q_2)}\right).
\end{split}
\end{equation}
Setting 
$$
h_1:=m_1+m_2, \quad h_2:=m_1-m_2, \quad j_1:=n_1+n_2,\quad j_2:=n_1-n_2,
$$
the right-hand side of \eqref{7} turns into 
\begin{equation} \label{8}
\begin{split}
& \sum\limits_{\alpha} \sum\limits_{\beta} e^{-4\pi (\alpha^2 +\beta^2)Q^2/N} \cdot
\mathop{\sum\limits_{h_1}\sum\limits_{h_2}}_{h_1\equiv h_2\bmod{2}} \mathop{\sum\limits_{j_1}\sum\limits_{j_2}}_{j_1\equiv j_2\bmod{2}} 
e^{-\pi (h_1^2+j_1^2+h_2^2+j_2^2)/(2Q)} \times\\ &
 e\left(\frac{(h_1h_2-j_1j_2)(\alpha k+\beta l)+(h_1j_2+h_2j_1)(-\beta k+\alpha l)}{\mathcal{N}(q_2)}\right)\\
 = &  
\mathop{\sum\limits_{h_1}\sum\limits_{h_2}}_{h_1\equiv h_2\bmod{2}} \mathop{\sum\limits_{j_1}\sum\limits_{j_2}}_{j_1\equiv j_2\bmod{2}} 
e^{-\pi (h_1^2+j_1^2+h_2^2+j_2^2)/(2Q)} \cdot \sum\limits_{\alpha} \sum\limits_{\beta} e^{-4\pi (\alpha^2 +\beta^2)Q^2/N} \times\\ &
 e\left(\frac{\alpha\left((h_1h_2-j_1j_2)k+(h_1j_2+h_2j_1)l\right)+\beta\left(-(h_1j_2+h_2j_1)k+(h_1h_2-j_1j_2)l\right)}{\mathcal{N}(q_2)}\right).
 \end{split}
\end{equation}

Applying the Poisson summation for the sums over $\alpha$ and $\beta$ and then changing variables into $a=h_1+ij_1$ and $b=h_2+ij_2$, the right-hand side of \eqref{8} transforms into
\begin{equation} \label{9}
\begin{split}
& \frac{N}{4Q^2}\cdot \mathop{\sum\limits_{h_1}\sum\limits_{h_2}}_{h_1\equiv h_2\bmod{2}} \mathop{\sum\limits_{j_1}\sum\limits_{j_2}}_{j_1\equiv j_2\bmod{2}} 
e^{-\pi (h_1^2+j_1^2+h_2^2+j_2^2)/(2Q)} \cdot \sum\limits_{\gamma} \sum\limits_{\delta}  \\ &
 e^{-\pi \left(\left(((h_1h_2-j_1j_2)k+(h_1j_2+h_2j_1)l)/\mathcal{N}(q_2)-\gamma\right)^2+\left((-(h_1j_2+h_2j_1)k+(h_1h_2-j_1j_2)l)/\mathcal{N}(q_2)-\delta\right)^2\right)N/(4Q^2)}\\
 = & \frac{N}{4Q^2} \cdot \mathop{\sum\limits_{a\in \mathbb{Z}[i]}\sum\limits_{b\in \mathbb{Z}[i]}}_{a\equiv b\bmod{2}} 
e^{-\pi (\mathcal{N}(a)+\mathcal{N}(b))/(2Q)} \times\\ & \sum\limits_{\gamma} \sum\limits_{\delta}  
e^{-\pi \left(\left((\Re(ab)k+\Im(ab)l)/\mathcal{N}(q_2)-\gamma\right)^2+
 \left((-\Im(ab)k+\Re(ab)l)/\mathcal{N}(q_2)-\delta\right)^2\right)N/(4Q^2)}.
 \end{split}
\end{equation}
Setting $ab=q'=u'+v'i$ and re-arranging summations, the last line turns into  
\begin{equation} \label{10}
\begin{split}
 & \frac{N}{4Q^2} \cdot \sum\limits_{q'\in \mathbb{Z}[i]} \sum\limits_{\gamma} \sum\limits_{\delta}  
e^{-\pi \left(\left((u'k+v'l)/\mathcal{N}(q_2)-\gamma\right)^2+
 \left((-v'k+u'l)/\mathcal{N}(q_2)-\delta\right)^2\right)N/(4Q^2)}\times\\ & \mathop{\sum\limits_{a\in \mathbb{Z}[i]}\sum\limits_{b\in \mathbb{Z}[i]}}_{\substack{a\equiv b\bmod{2}\\ ab=q'}} 
e^{-\pi (\mathcal{N}(a)+\mathcal{N}(b))/(2Q)}\\ 
= & \frac{N}{4Q^2} \cdot \Bigg( \sum\limits_{q'\in \mathbb{Z}[i]\setminus\{0\}} \sum\limits_{\gamma} \sum\limits_{\delta}  
e^{-\pi \left(\left((u'k+v'l)/\mathcal{N}(q_2)-\gamma\right)^2+
 \left((-v'k+u'l)/\mathcal{N}(q_2)-\delta\right)^2\right)N/(4Q^2)}\times\\ & \mathop{\sum\limits_{a\in \mathbb{Z}[i]}\sum\limits_{b\in \mathbb{Z}[i]}}_{\substack{a\equiv b\bmod{2}\\ ab=q'}} 
e^{-\pi (\mathcal{N}(a)+\mathcal{N}(b))/(2Q)} +
\sum\limits_{\gamma} \sum\limits_{\delta}  e^{-\pi \left(\gamma^2+\delta^2\right)N/(4Q^2)} \cdot \Bigg(2 \sum\limits_{c\in \mathbb{Z}[i]} e^{-2\pi\mathcal{N}(c)/Q}-1\Bigg)\Bigg).
 \end{split}
\end{equation}

Clearly, 
\begin{equation} \label{11}
\sum\limits_{\gamma} \sum\limits_{\delta}  e^{-\pi \left(\gamma^2+\delta^2\right)N/(4Q^2)} \cdot \left(2 \sum\limits_{c\in \mathbb{Z}[i]} e^{-2\pi\mathcal{N}(c)/Q}-1\right)
\ll \left(1+\frac{Q^2}{N}\right) \cdot Q \ll Q,
\end{equation}
and the geometric-arithmetic mean inequality gives
\begin{equation} \label{12}
\begin{split}
& \mathop{\sum\limits_{a\in \mathbb{Z}[i]}\sum\limits_{b\in \mathbb{Z}[i]}}_{\substack{a\equiv b\bmod{2}\\ ab=q'}} 
e^{-\pi (\mathcal{N}(a)+\mathcal{N}(b))/(2Q)} \le \mathop{\sum\limits_{a\in \mathbb{Z}[i]}\sum\limits_{b\in \mathbb{Z}[i]}}_{\substack{a\equiv b\bmod{2}\\ ab=q'}} 
e^{-\pi \sqrt{\mathcal{N}(a)\mathcal{N}(b)}/Q}\\
= & e^{-\pi \sqrt{\mathcal{N}(q')}/Q} \cdot \mathop{\sum\limits_{a\in \mathbb{Z}[i]}\sum\limits_{b\in \mathbb{Z}[i]}}_{\substack{a\equiv b\bmod{2}\\ ab=q'}} 1 \ll 
\mathcal{N}(q')^{\varepsilon} \cdot e^{-\pi \sqrt{\mathcal{N}(q')}/Q},
\end{split}
\end{equation}
where for the last inequality, we have used the estimate
$$
\sum\limits_{\substack{d\in \mathbb{Z}[i]\\ d|z}} 1 \ll \mathcal{N}(z)^{\varepsilon}
$$
for the generalized divisor function in $\mathbb{Z}[i]$.

Combining \eqref{5}, \eqref{6}, \eqref{7}, \eqref{8}, \eqref{9}, \eqref{10}, \eqref{11} and  \eqref{12}, and taking the square root, we obtain
\begin{equation*} 
\begin{split}
 \Sigma\left(Q^2,N;\mathcal{S}\right) \ll &
\left(Q^3+Q^{1/2}N\right)Z+NZ \cdot \max\limits_{\substack{\mathcal{N}(q_2)=\Box\\ Q^2/2<\mathcal{N}(q_2)\le Q^2\\ r_2 \bmod q_2\\ (r_2,q_2)=1}}\left|\sum\limits_{q'\in \mathbb{Z}[i]\setminus\{0\}} \mathcal{N}(q')^{\varepsilon} \cdot e^{-\pi \sqrt{\mathcal{N}(q')}/Q}\times\right. \\ & \left. 
\sum\limits_{\gamma} \sum\limits_{\delta}  
e^{-\pi \left(\left((u'k+v'l)/\mathcal{N}(q_2)-\gamma\right)^2+
 \left((-v'k+u'l)/\mathcal{N}(q_2)-\delta\right)^2\right)N/(4Q^2)}\right|^{1/2}.
\end{split}
\end{equation*}
Now using $Q\le N^{1/2-\varepsilon}$ and the rapid decay of the function $e^{-x^2}$, we may cut summations at the cost of a small error, leading to
\begin{equation} \label{repl} 
\begin{split}
&  \Sigma\left(Q^2,N;\mathcal{S}\right) \ll
\left(Q^3+Q^{1/2}N\right)Z+NQ^{\varepsilon}Z \times\\ & \max\limits_{\substack{\mathcal{N}(q_2)=\Box\\ Q^2/2<\mathcal{N}(q_2)\le Q^2\\ r_2 \bmod q_2\\ (r_2,q_2)=1}} \sharp\left\{q'\in \mathbb{Z}[i] : \mathcal{N}(q')\le Q^2N^{\varepsilon}, \    
\binom{f\left((u'k+v'l)/\mathcal{N}(q_2)\right)}{f\left((-v'k+u'l)/\mathcal{N}(q_2)\right)}\in D_{R'}({\bf 0})\right\}^{1/2},
\end{split}
\end{equation}
where $f(x)$ is defined as in \eqref{fdef} and 
$$
R':=4Q^2N^{\varepsilon-1}.
$$

Using Proposition \ref{ballcount} with $Q$ replaced by $Q^2$ and $L=N^{\varepsilon}$, the maximum on the right-hand side of \eqref{repl} is bounded by 
$$
\ll \left(N^{\varepsilon}\left(1+\frac{R'}{1/Q^2}\right)\right)^{1/2} =N^{\varepsilon/2}\left(1+4Q^{4}N^{\varepsilon-1}\right)^{1/2}.
$$
Hence, we have 
\begin{equation} \label{ab} 
\Sigma\left(Q^2,N;\mathcal{S}\right) \ll (QN)^{\varepsilon}
\left(Q^3+Q^{1/2}N+Q^2N^{1/2}\right)Z.
\end{equation}
Taking \eqref{ulm} into consideration, we deduce that \eqref{ab} holds
for all $Q,N\ge 1$. This together with \eqref{with} gives Theorem \ref{theo5} upon summing up the contributions of $O(\log 2Q)$ dyadic intervals containing the moduli norms $\mathcal{N}(q)$.

\section{Open problems}
The following problems appear naturally in connection with this work. \\ \\
(i) Can these results be extended to general number fields?\\ \\
(ii) What can be proved for more general sets of moduli such as moduli whose norms are represented by polynomials?\\ \\
(iii) Is it possible to improve the above large sieve inequality for square norm moduli along similar lines as  in  \cite{BaZh}?

\end{document}